\documentclass[11pt]{amsart}
\usepackage{amsthm}
\usepackage{amscd}  
\usepackage{graphicx}
\usepackage{amssymb}
\usepackage{epstopdf}
\theoremstyle{plain}
\newtheorem{Thm}{Theorem}[section]

\newtheorem{Prop}[Thm]{Proposition}

\theoremstyle{definition}

\newtheorem{Expl}[Thm]{Example}
\newtheorem{Rem}[Thm]{Remark}
 
\numberwithin{equation}{section}

\title{Deformations over non-commutative base}
\author{Yujiro Kawamata \\
Dedicated to the Memory of Professor Jean-Pierre Demailly}
\usepackage{fancyhdr}
\begin{document}
\maketitle

\tableofcontents

\begin{abstract}
We make some remarks on deformation theory over non-commutative base.
We describe the base algebra of semi-universal non-commutative deformations using vector spaces $T^1$ and $T^2$.
%The non-commutative deformations of compact complex manifolds are also treated.

14D15, 32G05.
\end{abstract}

We will consider deformation theory over non-commutative (NC) base algebras.
Such a theory is interesting because there are more deformations than the usual deformations over commutative bases.
The deformations over commutative base can possibly be regarded as the \lq first order' approximation 
of more general \lq higher order' deformations.
The formal theories of deformations over commutative and non-commutative bases 
are parallel and the extension to the non-commutative case is simple,
but some new phenomena and invariants appear.

We make some remarks on NC deformations.
The first remark is that the deformations over NC base is natural.
This is because the differential graded algebras (DGA) 
which govern the deformations of sheaves are naturally non-commutative.
Hence it is natural to consider deformations parametrized by NC base algebras.
We will also consider the problem of convergence of formal NC deformations and the moduli space.
The second remark is that we obtain \lq higher order invariants' because there are more NC deformations than
commutative ones by slightly generalizing results of \cite{Toda} and \cite{Hua-Toda}.
The last remark is that a description of the base algebra using the tangent space $T^1$ and 
the obstruction space $T^2$ is possible.
%We will also extend the deformation theory of compact complex manifolds to the NC case.

We use the abbreviation NC for \lq\lq not necessarily commutative''.
In \S 1, we recall the definition of NC deformations, and explain how the base algebra of semi-universal NC deformations
is described by a minimal $A^{\infty}$-algebra arising from DGA in the case of deformations of coherent sheaves.
In \S2, we consider the problem of convergence and the existence of moduli space by taking an example 
of deformations of linear subspaces in a linear space.
In \S 3, we consider another example of flopping contractions of $3$-dimensional manifolds, and show how
invariants appear beyond those obtained by commutative deformations.
We will give a description of the base algebra of the semi-universal NC deformation by using the tangent space and the 
obstruction space in \S 4.
%In \S 5, we will consider NC deformations of compact complex manifolds.
%We intended to allow non-commutativity to the base but ended up allowing non-commutativity to the fibers but not to the base. 
%We will prove that a description of a versal NC deformations using $T^1$ and $T^2$ given by Hochschild cohomologies.   

%The author would like to thank Michel Van den Bergh and Anya Nordskova for pointing out mistakes in \S5 
%of the first and the second versions of this paper, respectively, and kindly gave the author encouragements.
The author would like to thank Jungkai Alfred Chen and NCTS of National Taiwan University where 
the work was partly done while the author visited there.
He would also like to thank the referee for the careful reading and suggestions for improvements. 
This work is partly supported by JSPS Kakenhi 21H00970.

%%%%%%%%%%%%%%%%%%%%%%%%%%%%%%%%%%%%%
%%%%%%%%%%%%%%%%%%%%%%%%%%%%%%%%%%%%%
%%%%%%%%%%%%%%%%%%%%%%%%%%%%%%%%%%%%
\section{multi-pointed non-commutative deformations}

We recall the non-commutative deformation theory developed by \cite{Laudal} (see also \cite{ELS}, \cite{multi}).
We use NC as \lq\lq not necessarily commutative''. 
This is a generalization of the formal commutative deformation theory of \cite{Schlessinger} to the case 
where the base algebras are allowed to be NC.

Let $k^r$ be the direct product ring of a field $k$, and let $(\text{Art}_r)$ be the category of {\em augmented 
associative} $k^r$-algebras $R$
which are finite dimensional as $k$-modules and such that the two-sided ideal 
$M = \text{Ker}(R \to k^r)$ is nilpotent.
We assume that the composition of the structure homomorphisms $k^r \to R \to k^r$ is the identity.
$(\text{Art}_r)$ is the category of the base spaces for {\em $r$-pointed} NC deformations.

Let $k_i \cong k$ be the $i$-th direct factor of the product ring $k^r$ for $1 \le i \le r$.
$k_i$ is generated by $e_i = (0,\dots,1,\dots, 0) \in k^r$, where $1$ is placed at the $i$-th entry.
A left $k^r$-module $F$ has a direct sum decomposition $F = \bigoplus_{i=1}^r F_i$ as $k$-modules by $F_i = e_iF$, and 
$k^r$-bimodule has a further decomposition $F = \bigoplus_{i,j=1}^r F_{ij}$ by $F_{ij} = e_iFe_j$.

$R \in (\text{Art}_r)$ is an NC Artin semi-local algebra with maximal two-sided ideals $M_i = \text{Ker}(R \to k_i)$.
NC deformation is {\em multi-pointed} because an NC semi-local algebra is not necessarily a direct product of local algebras
unlike the case of a commutative algebra.

The model case is a deformation of a direct sum of coherent sheaves $F = \bigoplus_{i=1}^r F_i$ ($r$-pointed sheaf). 
The sheaves $F_i$ interact each other and there are more NC deformations of $F$ than those of the individual 
sheaves $F_i$.

Let $F$ be something defined over $k^r$ which will be deformed over $R \in (\text{Art}_r)$.
An {\em NC deformation} of $F$ over $R$ is a pair $(\tilde F, \phi)$ 
where $\tilde F$ is \lq\lq flat'' over $R$ 
and $\phi: F \to R/M \otimes_R \tilde F$ is an isomorphism.
The definition depends on the cases what kind of $F$ we are considering.
The set of isomorphism classes of deformations of $F$ over $R$ gives an {\em NC deformation functor} 
$\Phi = \text{Def}_F: (\text{Art}_r) \to (\text{Set})$.

More concretely, an {\em $r$-pointed NC deformation functor} $\Phi: (\text{Art}_r) \to (\text{Set})$ in this paper is
a covariant functor which satisfies 
the conditions $(H_0), (H_f), (H_e), (\tilde H)$ stated below 
following \cite{Schlessinger} (see also \cite{Sernesi} Chapter 2).

\vskip 1pc

We define an object $R_e \in (\text{Art}_r)$ as a generalization of the ring of dual numbers
$k[\epsilon]/(\epsilon^2)$.
Let $R_e$ be the trivial extension $k^r \oplus \text{End}(k^r)$, 
where $\text{End}(k^r)$ is a square zero two-sided ideal, and the 
multiplication of $k^r$ and $\text{End}(k^r)$ is induced from the embedding to diagonal matrices $k^r \to \text{End}(k^r)$.
As a $k$-module, 
\[
R_e = k^r \oplus \bigoplus_{i,j=1}^r ke_{ij}.
\]
The multiplication is defined by $e_ie_{jk} = \delta_{ij}e_{jk}$, $e_{ij}e_k = \delta_{jk}e_{ij}$ 
and $e_{ij}e_{kl} = 0$
for all $i,j,k,l$.
The augmentation $R_e \to k^r$ is given by $e_{ij} \mapsto 0$.

\vskip 1pc

Now we state the conditions $(H_0), (H_f), (H_e), (\tilde H)$.
For ring homomorphisms $R' \to R$ and $R'' \to R$ in $(\text{Art}_r)$, 
let $\alpha: \Phi(R' \times_R R'') \to \Phi(R') \times_{\Phi(R)} \Phi(R'')$ be the map naturally defined by $\Phi$.

\begin{itemize}

\item[$(H_0)$] $\Phi(k^r)$ consists of one element.

\item[$(H_f)$] $\Phi(R_e)$ is finite dimensional as a $k$-module.

\item[$(H_e)$] The natural map $\alpha$ is bijective if $R = k^r$ and $R'' = R_e$.

\item[$(\tilde H)$] The natural map $\alpha$ is surjective if $R'' \to R$ is surjective.

\end{itemize}

The {\em tangent space} $T^1$ of the functor $\Phi$ is defined by $T^1 = \Phi(R_e)$.
The $k^r$-bimodule structure of the ideal $\text{End}(k^r) \subset R_e$ induces 
a $k^r$-bimodule structure on $T^1$, so we can write $T^1 = \bigoplus_{i,j=1}^r T^1_{ij}$.
We have $T^1_{ij} = \Phi(k^r \oplus ke_{ij})$.
Indeed $T^1_{ij} = e_iT^1e_j = \Phi(e_iR_ee_j) = \Phi(k^r \oplus ke_{ij})$.

An element $\xi \in \Phi(R)$ for $R \in (\text{Art}_r)$ is called an {\em $r$-pointed NC deformation} over $R$ of
the unique element of $\Phi(k^r)$.

Let $T_R = (M/M^2)^*$ be the Zariski tangent space of $R$.
It is a $k^r$-bimodule.
The {\em Kodaira-Spencer map} $KS_{\xi}: T_R \to T^1$ associated to the deformation $\xi$ is defined
as follows.
A tangent vector $v \in (T_R)_{ij} = (M/M^2)^*_{ij}$ induces a ring homomorphism $v_*: R \to k^r \oplus ke_{ij}$, hence
$\Phi(v_*): \Phi(R) \to \Phi(k^r \oplus ke_{ij}) = T^1_{ij}$.
Then we define $KS_{\xi}(v) = \Phi(v_*)(\xi)$.

\vskip 1pc

Let $\hat R := \varprojlim R_i \in (\hat{\text{Art}}_r)$ be a pro-object of $(\text{Art}_r)$, and let 
$\hat{\xi} := \varprojlim \xi_i \in \hat{\Phi}(\hat R) := \varprojlim \Phi(R_i)$ be an element of a projective limit.
Then $\hat{\xi}$ is called a {\em formal $r$-pointed NC deformation} over $\hat R$.
The Kodaira-Spencer map $KS_{\hat{\xi}}: T_{\hat R} \to T^1$ is similarly defined.

A formal deformation $\hat{\xi} \in \hat{\Phi}(\hat R)$ is called a {\em versal} NC deformation if the following holds:
for any NC deformation $\xi' \in \Phi(R')$, there exists a morphism 
$h: \hat R \to R'$ such that $\xi' = \hat{\Phi}(h)(\hat{\xi})$.

In this case, the Kodaira-Spencer map $KS_{\hat{\xi}}: T_{\hat R} \to T^1$ is surjective.
Indeed, let $v' \in T^1_{ij} = \Phi(k^r \oplus ke_{ij})$ be any element.
Then there is a morphism $h: \hat R \to k^r \oplus ke_{ij}$ such that $v' = \hat{\Phi}(h)(\hat{\xi})$.
Let $v: (\hat M/\hat M^2)_{ij} \to ke_{ij}$ be the homomorphism induced from $h$.
Then $v_* = h$ and $KS_{\hat{\xi}}(v) = v'$. 

A versal NC deformation is said to be {\em semi-universal} if the Kodaira-Spencer map is bijective.
In this case, we have $\hat M/\hat M^2 \cong (T^1)^*$.
We note that it is called \lq\lq versal'' in some literatures.
The existence of the semi-universal NC deformation is proved in a similar way to \cite{Schlessinger}
from the conditions $(H_0), (H_f), (H_e), (\tilde H)$.

In the case $r = 1$, if we take the abelianization $\hat R^{ab} = \hat R/[\hat R, \hat R]$ of the base ring of the 
semi-universal deformation, then we 
obtain a usual semi-universal commutative deformation $\hat{\xi}^{ab}$ over $\hat R^{ab}$ given by 
$\hat{\xi}^{ab} = \Phi(q)(\hat{\xi})$, where $q: \hat R \to \hat R^{ab}$ is the quotient map.

\vskip 1pc

We recall a description of the semi-universal NC deformation in the case of deformations of a coherent sheaf 
using an $A^{\infty}$-algebra formalism (\cite{formal}).
Let $X$ be an algebraic variety over $k$ and 
let $F = \bigoplus_{i=1}^r F_i$ be a coherent sheaf with proper support.
Then the infinitesimal deformations of $F$ are controlled by
a {\em differential graded algebra (DGA)} $RHom_X(F, F)$.
The tangent space and the obstruction space are given 
by $k^r$-bimodules $T^i = \text{Ext}_X^i(F, F)$ for $i = 1,2$ (cf. \S 4).

It is also controlled by an $A^{\infty}$-algebra structure
$\{m_d\}_{d \ge 2}$ of the cohomology group 
$A = \bigoplus_{p \ge 0} A_p := \bigoplus_{p \ge 0} \text{Ext}^p(F,F) = \bigoplus_{p,i,j} \text{Ext}^p(F_i,F_j)$; 
\[
m_d: T^d_{k^r}A := A \otimes_{k^r} \dots \otimes_{k^r} A \to A(2-d)
\]
are the higher multiplications of degree $2-d$, where the left hand side is a tensor product with $d$ factors over $k^r$ and the 
right hand side has degree shift $2-d$.
In particular, we have 
\[
m_d: T^d_{k^r}A_1 := A_1 \otimes_{k^r} \dots \otimes_{k^r} A_1 \to A_2
\]
for $d \ge 2$.

In general, for a $k^r$-bimodule $E$, we have $E = \bigoplus_{i,j=1}^r E_{ij}$ with $E_{ij} = e_iEe_j$.
We define a completed tensor algebra $\hat T_{k^r}E = \prod_{d \ge 0} T^d_{k^r}E$ by
\[
T^d_{k^r} E = E \otimes_{k^r} E \otimes_{k^r} \dots \otimes_{k^r} E
\]
where there are $d$-times $E$ on the right hand side.
We apply this construction to $E = (T^1)^*$.
If $\{x_{ij}^s\}_s$ is a basis of $E_{ij}$, then we have 
\[
\hat T_{k^r}E = k^r\langle\langle x_{ij}^s \rangle\rangle/(e_ix_{i'j}^s, x_{ij}^se_{j'}, x_{i'j'}^sx_{i''j''}^{s'} 
\mid i \ne i', j \ne j', j' \ne i'').
\]
Thus the set of monomials
\[
x_{i_0i_1}^{s_1}x_{i_1i_2}^{s_2}\dots x_{i_{d-1}i_d}^{s_d} 
\]
with $i = i_0$ and $j = i_d$ is a $k$-basis of $(\hat T_{k^r}E)_{ij}$.

Let 
\[
m^* = \sum_{d \ge 2} m_d^*: \text{Ext}^2(F,F)^* \to \hat T_{k^r}(\text{Ext}^1(F,F)^*)
\]
be the formal sum of dual maps of $m_d$.
Then the base algebra $\hat R$ of the semi-universal NC deformation $\hat F$ is determined as an augmented $k^r$-algebra 
to be
\[
\hat R = \hat T_{k^r}(\text{Ext}^1(X,X)^*)/(m^*(\text{Ext}^2(X,X)^*))
\]
(\cite{formal}).
Thus the Taylor coefficients of the equations of the formal NC moduli space are determined by $A^{\infty}$-multiplications.

\vskip 1pc

There is another way of describing a semi-universal $r$-pointed NC deformation of a direct sum 
of coherent sheaves with proper support $F = \bigoplus_{i=1}^r F_i$.
The semi-universal NC deformation $\hat F$ of $F$ is given by a tower $\{F^{(n)}\}$ of universal extensions (cf. \cite{multi}):
\[
0 \to \text{Ext}^1(F^{(n)}, F)^* \otimes_{k^r} F \to F^{(n+1)} \to F^{(n)} \to 0
\]
with $F^{(0)} = F$ and $\hat F = \varprojlim F^{(n)}$.
We have direct sum decompositions $F^{(n)} = \bigoplus_i F^{(n)}_i$, and we can write
\[
0 \to \bigoplus_{i,j} \text{Ext}^1(F^{(n)}_i, F_j)^* \otimes_k F_j \to \bigoplus_i F^{(n+1)}_i \to \bigoplus_i F^{(n)}_i \to 0.
\]

If $\text{End}(F) \cong k^r$, i.e., if $\text{End}(F_i) \cong k$ and $\text{Hom}(F_i,F_j) \cong 0$ for $i \ne j$, then 
$F$ is called a {\em simple collection} (\cite{multi}).
The deformation theory of a simple collection is particularly nice.
In this case, $F^{(n)}$ is flat over $R^{(n)} = \text{End}(F^{(n)})$, and 
the parameter algebra $\hat R$ of the semi-universal deformation $\hat F$ is given by
$\hat R = \varprojlim R^{(n)}$ (\cite{multi} Theorem 4.8).

\begin{Rem}
(1) We do not consider deformation theory of {\em varieties} over non-commutative base in this paper,
because such a theory seems to be difficult by the following reason.
Suppose that there is an infinitesimal deformation $X_R$ of a variety $X$ over an NC ring $R$.
Then the structure sheaf $\mathcal O_{X_R}$ should be NC too.
When we consider a base change over a ring homomorphism $R \to R'$, 
it seems necessary that the base rings should be commutative
in order for the tensor product $\mathcal O_{X_R} \otimes_R R'$ to have a ring structure.
Indeed the DGLie algebra which controls the deformations of $X$ is NC but its non-commutativity is restricted.

But when $X$ is a subvariety of an ambient variety $Y$, then we can consider a deformation of $X$ inside $Y$ over an NC base
as a deformation of the structure sheaf $\mathcal O_X$ as a sheaf on $Y$ (see \S 2).

(2) The deformation functor is pro-representable when there is a universal deformation.
But a universal deformation does not exist in general (see \cite{multi} Remark 4.10).
\end{Rem}

%%%%%%%%%%%%%%%%%%%%%%%%%%%%%%
%%%%%%%%%%%%%%%%%%%%%%%%%%%%%%
%%%%%%%%%%%%%%%%%%%%%%%%%%%%%%
\section{convergence and moduli}

The above described semi-universal NC deformation is a formal deformation, and the question on the convergence 
is important.
We will make some remarks on the convergence of the formal NC deformations and the relationship with
the moduli space of commutative deformations. 
We consider only $1$-pointed NC deformations, and 
we take an example of the moduli space of linear subspaces in a fixed linear space.
We consider NC deformations of the structure sheaves of linear subspaces.

We would like to say that the formal semi-universal NC deformation is convergent 
if the corresponding semi-universal commutative deformation is convergent.
This is because the numbers of commutative monomials and non-commutative ones on $n$ variables of degree $d$
grow similarly to $n^d$. 
Maybe we should require that the growth of the Taylor coefficients of the non-commutative power series 
are bounded in a similar way as the commutative power series. 

Any $k$-algebra homomorphism $R \to k$ for any associative $k$-algebra $R$
factors through the abelianization $R \to R^{ab}$.
Therefore we can think that the set of closed points of the moduli spaces are the same 
for commutative and NC deformation problems. 
In other words, when we observe points, then the moduli space of NC deformations
is reduced to the usual moduli space.
We can say that the NC deformations give an additional infinitesimal or 
formal structure at each point of the commutative moduli space.
And the formal structure is usually convergent.
However, a compactification is another problem, and it seems that it does not exists.

\vskip 1pc

As an example, we consider NC deformations of linear subspaces in a finite dimensional vector space.
As explained in Remark 1.1, we consider the NC deformations of the structure sheaf of the subspace
instead of the subspace as a variety.
The following is a slight generalization of \cite{formal} Example 7.8.
The commutative deformations are unobstructed and yield a compact moduli space, a Grassmann variety.
But we will see that NC deformations are obstructed.

Let $V \cong k^n$ be an $n$-dimensional linear space with coordinate linear functions $x_1,\dots,x_n$, and
let $W$ be an $m$-dimensional linear subspace defined by an ideal $I = (x_{m+1}, \dots, x_n)$.
The commutative moduli space $G(m,n)$ has an affine open subset $Hom(W, V/W) \cong k^{m(n-m)}$ 
with coordinates $a_{i,j}$ ($1 \le i \le m$, $m+1 \le j \le n$).
We consider NC deformations of $W$ as a linear subspace of $V$, i.e., the NC deformations of 
the ideal sheaves generated by linear functions.

\begin{Prop}
Let $V \cong k^n$ with coordinate linear functions $x_1,\dots,x_n$, 
and let $W \cong k^m$ be defined by $x_{m+1} = \dots = x_n = 0$.
Then the formal semi-universal NC deformation of $W$ as a linear subspace of $V$ has the parameter algebra $\hat R$ 
and the ideal $\hat I$ given as follows:
\[
\begin{split}
&\hat R = k\langle \langle a_{ij} \mid 1 \le i \le m < j \le n \rangle \rangle/\hat J \\
&\hat J = (a_{ij_1}a_{ij_2} - a_{ij_2}a_{ij_1}, \,\,\,
a_{i_1j_1}a_{i_2j_2} - a_{i_2j_2}a_{i_1j_1} + a_{i_1j_2}a_{i_2j_1} - a_{i_2j_1}a_{i_1j_2} \\
&\mid 0 \le i \le m, 1 \le i_1 < i_2 \le m < j_1 < j_2 \le n) \\
&\hat I = (x_j + \sum_{i=1}^m a_{ij}x_i \mid m+1 \le j \le n).
\end{split}
\]
\end{Prop}

\begin{proof}
This is almost the same as \cite{formal} Example 7.8.
Let $Y = \mathbf P(W^*) \subset X = \mathbf P(V^*)$ be the corresponding projective spaces.
We consider NC deformations of a coherent sheaf $F = \mathcal O_Y$ on $X$.
The normal bundle of $Y$ in $X$ is given by $N_{Y/X} \cong \mathcal O_Y(1)^{\oplus n-m}$.
Hence $T^1 = \text{Ext}^1(F,F) \cong H^0(Y,N_{Y/X}) \cong k^{\oplus m(n-m)}$ and
$T^2 = \text{Ext}^2(F,F) \cong H^0(Y,\bigwedge^2 N_{Y/X}) \cong k^{\oplus \binom{m+1}2 \binom{n-m}2}$. 

Let $I' = \mathcal O_X(-Y)$ be the ideal sheaf of $Y \subset X$ generated by the homogeneous coordinates 
$x_{m+1}, \dots, x_n$.
By \cite{formal} Lemma 7.6, the semi-universal NC deformation of $F$ is given in the form 
\[
\hat F = \varprojlim (R_n \otimes \mathcal O_X)/I'_n
\]
where $(R_n, M_n) \in (\text{Art}_1)$ such that $M_n^{n+1} = 0$.
By the flatness, the ideal sheaf $I'_n$ is generated by linear forms 
$x_j + \sum_{i=1}^m a_{ij}x_i$
for $m+1 \le j \le n$, where $a_{ij} \in M_n$.

Since the $x_i$ are commutative variables in $R_n \otimes \mathcal O_X$, we have $x_jx_l = x_lx_j$ for 
$m+1 \le j,l \le n$.
Hence equalities
\[
\sum_{i,k=1}^m a_{ij}a_{kl}x_ix_k = \sum_{i,k=1}^m a_{kl}a_{ij}x_ix_k
\]
hold in $F_n = (R_n \otimes \mathcal O_X)/I'_n$ for such $j,l$. 
It follows that
\[
\begin{split}
&a_{ij}a_{il} - a_{il}a_{ij} = 0 \quad (1 \le i \le m < j < l \le n), \\
&a_{ij}a_{kl} - a_{kl}a_{ij} + a_{kj}a_{il} - a_{il}a_{kj} = 0 \quad (1 \le i < k \le m < j < l \le n)
\end{split}
\]
in $\hat R = \varprojlim R_n$.
The above relations are non-commutative polynomials which are linearly independent quadratic forms, and 
their number is equal to 
\[
m \binom {n-m}2 + \binom m2 \binom {n-m}2 = \binom{m+1}2 \binom{n-m}2. 
\]
This is equal to the dimension of the obstruction space.
Therefore there are no more independent relations contained in $\hat J$.
\end{proof}

The above deformation is \lq\lq algebraizable''.
There is an NC deformation of ideals $\tilde I$ over a parameter algebra $\tilde R$ which is a 
quotient algebra of an NC polynomial algebra:
\[
\begin{split}
&\tilde R = k\langle a_{ij} \mid 1 \le i \le m < j \le n \rangle/\tilde J \\
&\tilde J = (a_{ij_1}a_{ij_2} - a_{ij_2}a_{ij_1}, \,\,\, a_{i_1j_1}a_{i_2j_2} - a_{i_2j_2}a_{i_1j_1} 
+ a_{i_1j_2}a_{i_2j_1} - a_{i_2j_1}a_{i_1j_2} \\
&\mid 1 \le i \le m, 1 \le i_1 < i_2 \le m < j_1 < j_2 \le n) \\
&\tilde I = (x_j + \sum_{i=1}^m a_{ij}x_i \mid m+1 \le j \le n)
\end{split}
\]
The meaning of this formula is that it induces a semi-universal NC deformation at every closed point of an affine open subset 
$\text{Spec}(\tilde R^{ab}) \subset G(m+1,n+1)$
with $\tilde R^{ab} = k[a_{ij} \mid 0 \le i \le m < j \le n]$.
Indeed we have 
\[
(a_{ij}-a_{ij}^0)(b_{kl}-b_{kl}^0) - (b_{kl}-b_{kl}^0)(a_{ij}-a_{ij}^0) = a_{ij}b_{kl} - b_{kl}a_{ij}
\]
for NC variables $a_{ij}, b_{kl}$ and $a_{ij}^0,b_{kl}^0 \in k$.

Hilbert schemes and Quot schemes are constructed from Grassmann varieties.
We wonder if their NC deformations are also semi-globalizable.

\begin{Expl}
(1) $n = 3$ and $m = 1$.
We have $G(1,3) \cong \mathbf P^2$.
Then $\tilde R \cong k\langle a,b \rangle/(ab-ba) = k[a,b]$.

(2) $n = 3$ and $m = 2$.
We have $G(2,3) \cong \mathbf P^2$.
Then $\tilde R = k\langle a,b \rangle$ is not Noetherian.
Indeed a two-sided ideal $(ab^ka \mid k > 0)$ is not finitely generated.

$\tilde R$ has a following quotient algebra, which corresponds to an NC deformation which is not semi-universal:
\[
R_{\epsilon} = k\langle a,b \rangle/(ab-ba-\epsilon)
\]
where $\epsilon \in k$.
For example, if $\epsilon = 1$, then $R_1 \cong k[t,d/dt]$.

(3) $n = 4$ and $m = 2$.
We have $G(2,4)$.
Then we have 
\[
\tilde R = k\langle a,b,c,d \rangle/(ab-ba,\,\,cd-dc,\,\,ad-da-bc+cb).
\] 
$\tilde R$ has a following quotient algebra:
\[
R_{\epsilon_1, \epsilon_2} = k\langle a,b,c,d \rangle/(ab-ba,\,\,cd-dc,\,\,ad-da-1,\,\,bc-cb-1, \,\,ac-ca-\epsilon_1, \,\,bd-db-\epsilon_2)
\]
where $\epsilon_i \in k$.
For example, if $\epsilon_i = 0$, then $R_1 \cong k[t_1,t_2,\partial/\partial t_1, \partial/\partial t_2]$.
\end{Expl}

%%%%%%%%%%%%%%%%%%%%%%%%%%%%%%%%%%%%%
%%%%%%%%%%%%%%%%%%%%%%%%%%%%%%%%%%%%%
%%%%%%%%%%%%%%%%%%%%%%%%%%%%%%%%%%%%%
\section{flopping contractions of $3$-folds}

As a typical example of multi-pointed NC deformations, 
we will consider NC deformations of exceptional curves of a flopping contraction from 
a smooth 3-fold $f: Y \to X$ over $k = \mathbf C$.
\cite{DW} observed that there are 
more NC deformations than commutative ones, and the base algebra of NC deformations 
gives an important invariant of the flopping contraction called the contraction algebra.
Indeed Donovan and Wemyss conjectured that the contraction algebra, which is a finite dimensional associative algebra, 
determines the complex analytic type of the singularity of $X$.
\cite{Toda} and \cite{Hua-Toda} proved that the dimension count of the contraction algebra yields 
Gopakumar-Vafa invariants of rational curves defined in \cite{Katz}.
We will consider slight generalizations where there are more than one exceptional curves.

Let $f: Y \to X = \text{Spec}(B)$ be a projective birational morphism defined over $k = \mathbf C$ 
from a smooth $3$-dimensional variety $Y$ whose exceptional locus $C$ is $1$-dimensional.
Let $C = \bigcup_{i=1}^r C_i$ be a decomposition into irreducible components.
We assume that $f$ is crepant, i.e., $(K_Y,C_i) = 0$ for all $i$.
It is known that $C_i \cong \mathbf P^1$, the dual graph of the $C_i$ is a tree, and 
$X$ has only isolated hypersurface singularities of multiplicity $2$.

The {\em contraction algebra} $R$ for $f$ is defined to be the base algebra of the semi-universal $r$-pointed NC 
deformation of the sheaf 
$F = \bigoplus_{i=1}^r \mathcal O_{C_i}(-1)$. 

\vskip 1pc

We consider commutative one parameter deformation of the contraction morphism $f: Y \to X$, and 
investigate the behavior of the contraction algebras under deformation.
Let $p: \mathcal X \to \Delta$ be a one parameter flat deformation of $X$ over a disk $\Delta$, and assume that there is 
a flat deformation $\tilde f: \mathcal Y \to \mathcal X$ of the flopping contraction $f: Y \to X$.
We assume that there are Cartier divisors $\mathcal L_1,\dots, \mathcal L_r$ on $\mathcal Y$ 
such that $(\mathcal L_i, C_j) = \delta_{i,j}$.
This is always achieved when we replace $X$ by its complex analytic germ containing $f(C)$ and 
$\Delta$ by a smaller disk.

Let $C^t = \bigcup_{j=1}^{s_t} C^t_j$ be the exceptional curves with decomposition to irreducible components 
for the flopping contraction $f_t: Y_t \to X_t$ for $t \ne 0$, where $Y_t = (p\tilde f)^{-1}(t)$ and $X_t = p^{-1}(t)$.
It is not necessarily connected even if $C$ is connected.
We may assume that $s = s_t$ is constant on $t \ne 0$.
We define integers $m_{j,i}$ by the degeneration of $1$-cycles $C^t_j \to \sum m_{j,i}C_i$ when $t \to 0$.
This means that $\mathcal O_{C^t_j}$ degenerates in a flat family to $\mathcal O_{\sum_i m_{j,i}C_i}$. 
We have $(\mathcal L_i,C^t_j) = m_{j,i}$.

If the deformation $\tilde f$ is generic, then 
$C^t$ is a disjoint union of $(-1,-1)$-curves, i.e., smooth rational curves whose normal 
bundles are isomorphic to $\mathcal O_{\mathbf P^1}(-1)^{\oplus 2}$.
In this case, we denote 
\[
m_j = \sum_i m_{j,i}, \quad n_d = \# \{j \mid m_j = d\}.
\]
The numbers $n_d$ should be called the {\em Goparkumar-Vafa invariants} (\cite{Katz} for the case $r = 1$). 
In the case $r = 1$, \cite{Toda} proved that $n_1$ is equal to the dimension of the abelianization of the contraction algebra
$n_1 = \dim R^{ab}$, while higher terms $n_d$ for $d \ge 2$ contribute to $\dim R$ (see Theorem \ref{flop} (3)). 

\vskip 1pc

We consider NC deformations of $F = \bigoplus_{i=1}^r F_i$ for $F_i = \mathcal O_{C_i}(-1)$ on $Y$ and $\mathcal Y$.
The set $\{F_i\}$ is called a {\em simple collection} on $Y$ and $\mathcal Y$ in the terminology of \cite{multi}
in the sense that $\text{Hom}_Y(F,F) \cong \text{Hom}_{\mathcal Y}(F,F) \cong k^r$.
The NC deformations of a simple collection behave particularly nice.

Let $\hat{\Delta} = \text{Spec}(k[[t]])$ be the completion of $\Delta$ at the origin.
By the flat base change $\hat{\Delta} \to \Delta$, we define
$\hat{\mathcal X} = \mathcal X \times_{\Delta} \hat{\Delta}$ and
$\hat{\mathcal Y} = \mathcal Y \times_{\Delta} \hat{\Delta}$.
Let $\hat f: \hat{\mathcal Y} \to \hat{\mathcal X}$ and $\hat p: \hat{\mathcal X} \to \hat{\Delta}$ be natural morphisms.

Let $\hat{\mathcal F} = \bigoplus_{i=1}^r \mathcal F_i$ and $\tilde F^0 = \bigoplus_{i=1}^r \tilde F_i^0$ 
be the semi-universal NC deformations of $F$ on $\hat{\mathcal Y}$ and $Y$, respectively, and let 
$\hat{\mathcal R}$ and $R$ be the base algebras of these semi-universal deformations.
We note that $\tilde F^0$ is obtained by finite number of extensions of the $F_i$ while $\hat{\mathcal F}$ may not.
This is because $C$ is isolated in $Y$ while $C$ may move inside $\mathcal Y$.
Hence we have $\dim R < \infty$ as $k$-modules.
We will see that $\dim \hat{\mathcal R} = \infty$ (see Theorem \ref{flop} (1)).

$\hat{\mathcal F}$ is also a semi-universal NC deformation of $F$ on $\mathcal Y$.
We will see that there is also a \lq\lq convergent version'' $\mathcal F$ on $\mathcal Y$, and 
$\hat{\mathcal F}$ is its completion.

By \cite{multi} Theorem 4.8, the base algebras coincide with the endomorphism algebras:
\[
\hat{\mathcal R} = \text{End}_{\hat{\mathcal Y}}(\hat{\mathcal F}), \quad R = \text{End}_Y(\tilde F).
\]

$\hat{\mathcal F}$ and $\tilde F^0$ can be described explicitly in the following way
(\cite{DW}, \cite{multi}, \cite{perv}).
In particular, there exists a sheaf $\mathcal F$ on $\mathcal Y$ such that 
\begin{equation}\label{flat}
\hat{\mathcal F} \cong \mathcal F \otimes_{\mathcal O_{\mathcal Y}} \mathcal O_{\hat{\mathcal Y}}
\end{equation}
i.e., the semi-universal NC deformation $\hat{\mathcal F}$ is convergent when we replace $\Delta$ by a smaller disk if necessary.

\vskip 1pc

By \cite{VdB}, we construct extensions of locally free sheaves on $\mathcal Y$:
\[
0 \to \mathcal O_{\mathcal Y}^{s_i} \to M_i \to \mathcal L_i \to 0
\]
with some integers $s_i$ such that $R^1\tilde f_*M_i^* = 0$, where $M_i^*$ is the dual sheaf.
Let $M = \bigoplus_{i=1}^r M_i$ and $M^0 = M \otimes_{\mathcal O_{\mathcal Y}} \mathcal O_Y$.
We also denote $\hat M = M \otimes_{\mathcal O_{\mathcal Y}} \mathcal O_{\hat{\mathcal Y}}$.
We have an exact sequence
\[
0 \to M^* \to M^* \to (M^0)^* \to 0.
\]
Since the dimensions of fibers of $\tilde f$ are at most $1$, we obtain $R^1f_*(M^0)^* = 0$ from $R^1\tilde f_*M_i^* = 0$.
Then semi-universal NC deformations $\hat{\mathcal F} = \bigoplus \hat{\mathcal F}_i$ and $\tilde F^0$ are 
given as the kernels of natural homomorphisms 
(\cite{perv} Theorem 1.2):
\[
\begin{split}
&0 \to \hat{\mathcal F} \to \hat f^*\hat f_*\hat M \to \hat M \to 0, \\
&0 \to \tilde F^0 \to f^*f_*M^0 \to M^0 \to 0.
\end{split}
\]
We define $\mathcal F$ by an exact sequence 
\[
0 \to \mathcal F \to \tilde f^*\tilde f_*M \to M \to 0
\]
and let $\mathcal R = \text{End}_{\mathcal Y}(\mathcal F)$.
By the flat base change, we obtain (\ref{flat}) and 
\[
\hat{\mathcal R} \cong \mathcal R \otimes_{\mathcal O_{\mathcal Y}} \mathcal O_{\hat{\mathcal Y}}.
\]
We denote $\tilde F^t = \mathcal F \otimes_{\mathcal O_{\mathcal Y}} \mathcal O_{Y_t}$ 
and $R^t = \mathcal R \otimes_{\mathcal O_{\Delta}} k_t$, where $Y_t = (p\tilde f)^{-1}(t)$ and
$k_t$ is the residue field at $t \in \Delta$. 

The following is a slight generalization of results in \cite{Toda} and \cite{Hua-Toda}: 
 
\begin{Thm}\label{flop}
(1) $\mathcal F$ is flat over $\Delta$, and $\tilde F^0 = \mathcal F \otimes_{\mathcal O_{\mathcal Y}} \mathcal O_Y$.

(2) (\cite{Hua-Toda} Conjecture 4.3). 
$\mathcal R$ is a flat $\mathcal O_{\Delta}$-module, 
and $R \cong \mathcal R \otimes_{\mathcal O_{\Delta}} k$, 
where $k$ is the residue field of $\mathcal O_{\Delta}$ at $0$.

(3) Assume in addition that $C^t$ is a disjoint union of $(-1,-1)$-curves $C^t_j$ for $t \ne 0$.
Then 
\[
\begin{split}
&\tilde F^t \cong \bigoplus_j \mathcal O_{C^t_j}(-1)^{m_j}, \\ 
&R^t \cong \prod_j \text{Mat}(m_j \times m_j), \\ 
&\dim R = \sum_j m_j^2 = \sum_d n_d d^2. 
\end{split}
\]
\end{Thm}

\begin{proof}
(1) We have an exact sequence 
\[
0 \to M \to M \to M^0 \to 0
\]
where the first arrow is the multiplication by $t$.
Because $R^1\tilde f_*M = 0$, there is an exact sequence
\[
0 \to \tilde f_*M \to \tilde f_*M \to \tilde f_*M^0 \to 0.
\] 
Because $L_1\tilde f^*\tilde f_*M^0 = 0$ by \cite{BB} Lemma 3.4, we obtain the first row of the 
following commutative diagram
\[
\begin{CD}
0 @>>> \tilde f^*\tilde f_*M @>>> \tilde f^*\tilde f_*M @>>> \tilde f^*\tilde f_*M^0 @>>> 0 \\
@. @VVV @VVV @VVV \\
0 @>>> M @>>> M @>>> M^0 @>>> 0.
\end{CD}
\]
By snake lemma, we obtain 
\[
0 \to \mathcal F \to \mathcal F \to \tilde F^0 \to 0
\]
hence the flatness.

\vskip 1pc

(2) Since $t: \mathcal F \to \mathcal F$ is injective, $\mathcal R$ has no $t$-torsion.
Thus it is sufficient to prove that the natural homomorphism 
$\text{Hom}_{\mathcal Y}(\mathcal F,\mathcal F) \to \text{Hom}_Y(\tilde F^0,\tilde F^0)$ is surjective.
By the flat base change, it is also sufficient to prove that 
$\text{Hom}_{\hat{\mathcal Y}}(\hat{\mathcal F},\hat{\mathcal F}) \to \text{Hom}_Y(\tilde F^0,\tilde F^0)$ is surjective, i.e.,
$\hat{\mathcal R} \to R$ is surjective.
Then the assertion follows from the fact that 
$\hat{\mathcal R}$ and $R$ are the base algebras of NC semi-universal deformations of the same sheaf $F$ with 
$Y \subset \mathcal Y$.

\vskip 1pc

(3) This is proved in \cite{Toda} and \cite{Hua-Toda} when $r=1$. 
Let $x^t_j = \tilde f(C^t_j) \in X_t = p^{-1}(t)$ for $t \ne 0$.
Since $C^t_j$ is a $(-1,-1)$-curve, $x^t_j$ is an ordinary double point on a $3$-fold.
We take a small complex analytic neighborhood $x^t_j \in U^t_j \subset X_t$, and let $V^t_j = \tilde f^{-1}(U^t_j)$.

Let $L^t_j$ be a Cartier divisor on $V^t_j$ such that $(L^t_j,C^t_j) = 1$.
We know that $(\mathcal L_i,C^t_j) = m_{j,i}$ and $R^1\tilde f_*M_i^* = 0$. 
Since $C^t_j \cong \mathbf P^1$ and $M_i$ is relatively generated, $M_i \vert_{V^t_j}$ is a direct sum of line bundles
whose degrees are non-negative but at most $1$.
Since the total degree is equal to $m_{j,i}$, 
it follows that $M_i \vert_{V^t_j} = (L^t_j)^{\oplus m_{j,i}} \oplus \mathcal O_{V^t_j}^{\oplus (\text{rank}(M_i)-m_{j,i})}$.

We will prove that $\text{Ker}(\tilde f^*\tilde f_*L^t_j \to L^t_j) \cong \mathcal O_{C^t_j}(-1)$.
Indeed there is a commutative diagram 
\[
\begin{CD}
@. \tilde f^*\tilde f_*(L^t_j)^* @>>> \mathcal O_{V^t_j}^{\oplus 2} @>>> \tilde f^*\tilde f_*L^t_j @>>> 0 \\
@. @V{h_1}VV @V{\cong}VV @V{h_2}VV \\
0 @>>> (L^t_j)^* @>>> \mathcal O_{V^t_j}^{\oplus 2} @>>> L^t_j @>>> 0.
\end{CD}
\]
Hence $\text{Ker}(h_2) \cong \text{Coker}(h_1)$.
Since $(L^t_j)^* \otimes_{\mathcal O_{V^t_j}} I_{C^t_j}$ for the ideal sheaf $I_{C^t_j}$ of $C^t_j \subset V^t_j$ 
is generated by global sections, 
we have $\text{Coker}(h_1) \cong (L^t_j)^* \otimes \mathcal O_{C^t_j} \cong \mathcal O_{C^t_j}(-1)$.

Therefore $\mathcal F_i \vert_{V^t_j} = \mathcal O_{C^t_j}(-1)^{\oplus m_{j,i}}$.
Hence $\mathcal F \vert_{V^t_j} = \mathcal O_{C^t_j}(-1)^{\oplus m_j}$, and 
$\tilde F^t \cong \bigoplus_j \mathcal O_{C^t_j}(-1)^{m_j}$.
Thus $\text{End}_{Y_t}(\tilde F^t) \cong \prod_j \text{Mat}(m_j \times m_j)$, and the assertion is proved.
\end{proof}

%%%%%%%%%%%%%%%%%%%%%%%%%%
%%%%%%%%%%%%%%%%%%%%%%%%%%
%%%%%%%%%%%%%%%%%%%%%%%%%%
\section{abstract description using $T^1$ and $T^2$}

We will describe the base algebra of the semi-universal NC deformation of a deformation functor $\Phi$ which
has the tangent space $T^1$ and the obstruction space $T^2$, which is defined below. 

Let $\Phi: (\text{Art}_r) \to (\text{Set})$ be an NC deformation functor 
which has a formal semi-universal deformation $\hat{\xi} \in \hat{\Phi}(\hat R)$.  
A $k^r$-bimodule $T^2 = \bigoplus_{i,j=1}^r T^2_{ij}$ is said to be the {\em obstruction space} 
if the following condition is satisfied.
Let $\xi \in \Phi(R)$ be an NC deformation over $(R,M) \in (\text{Art}_r)$, and 
let $(R',M') \in (\text{Art}_r)$ be an extension of $R$ by a two-sided ideal $J$:
\[
0 \to J \to R' \to R \to 0
\]
such that $M'J = 0$, so that $J$ is a left $k^r$-module.
Then there is an obstruction class $o_{\xi} \in T^2 \otimes_{k^r} J$ such that $\xi$ 
extends to an NC deformation $\xi' \in \Phi(R')$
if and only if $o_{\xi} = 0$.

We assume that the obstruction class is functorial in the following sense.
Let 
\begin{equation}\label{functorial}%%%
\begin{CD}
0 @>>> J @>>> R' @>>> R @>>> 0 \\
@. @VgVV @V{f'}VV @VfVV \\
0 @>>> J_1 @>>> R'_1 @>>> R_1 @>>> 0
\end{CD}
\end{equation}
be a commutative diagram of such extensions.
Let $\xi \in \Phi(R)$ be an NC deformation, 
and let $\xi_1 = \Phi(f)(\xi) \in \Phi(R_1)$.
Let $o_{\xi} \in T^2 \otimes_{k^r} J$ and $o_{\xi_1} \in T^2 \otimes_{k^r} J_1$ be the 
obstructions classes of extending $\xi$ and $\xi_1$ over $R'$ and $R'_1$, respectively.
Then $o_{\xi_1} = g(o_{\xi})$.

\begin{Thm}\label{relations}
Let $\Phi: (\text{Art}_r) \to (\text{Set})$ be an NC deformation functor.
Assume that the obstruction space $T^2$ is finite dimensional.
Then there is a $k^r$-linear map $m: (T^2)^* \to \hat T_{k^r}(T^1)^*$ such that 
$\hat R \cong  \hat T_{k^r}(T^1)^*/(m((T^2)^*))$, a quotient algebra of the completed tensor algebra by a two-sided ideal 
generated by the image of $m$.
\end{Thm}

\begin{proof}
Denote $\hat A = \hat T_{k^r}(T^1)^* = k^r \oplus \hat M$.
Then the base algebra of the semi-universal NC deformation $\hat R$ is a quotient algebra 
$\hat A/\hat I$ by some two-sided ideal $\hat I$.
Let $\{z_l\}_{l=1}^N$ be a $k$-basis of $T^2$.

Let $R_k = \hat A/(\hat I+\hat M^{k+1})$.
We define a sequence of two-sided ideals $I_k \subset \hat A/\hat M^{k+1}$ 
by $R_k = \hat A/(I_k + \hat M^{k+1})$.
By definition of the semi-universal deformation, there is an NC deformation $\xi_k \in \Phi(R_k)$.
We will prove that $I_k$ is generated by elements $\{s_{k,l}\}_{l=1}^N \in \hat A/\hat M^{k+1}$ such that 
$s_{k+1,l} \mapsto s_{k,l}$ by the natural map $\hat A/\hat M^{k+2} \to \hat A/\hat M^{k+1}$ inductively as follows. 

We set $s_{1,l} = 0$ for all $l$, because $I_1 = 0$ and $R_1 = \hat A/\hat M^2$.

Let $k$ be an arbitrary integer, and let $R = R_k$, 
$R' = \hat A/(\hat M \hat I+\hat M^{k+1})$ and $J = (\hat I+\hat M^{k+1})/(\hat M \hat I+\hat M^{k+1})$.
Then $R = R'/J$ and $M'J = 0$ for $M' = \hat M/(\hat M \hat I+\hat M^{k+1})$.
We write the obstruction of extending $\xi_k$ to $R'$ as $o_{\xi_k} = \sum_l z_l \otimes s_{k,l} \in T^2 \otimes_{k^r} J$, where 
$s_{k,l} \in J$.

We have a commutative diagram
\[
\begin{CD}
0 @>>> J @>>> R' @>>> R @>>> 0 \\
@. @VVV @VVV @V=VV \\
0 @>>> J/(s_{k,l}) @>>> R'/(s_{k,l}) @>>> R @>>> 0
\end{CD}
\]
By the functoriality of the obstruction class, the obstruction class of the lower sequence vanishes, and
$\xi_k$ is extendible to $R'/(s_{k,l})$.
By the semi-universality, it follows that 
\[
\hat I+\hat M^{k+1} = (s_{k,l}) + \hat M \hat I + \hat M^{k+1}.
\]
By Nakayama's lemma, we have $\hat I + \hat M^{k+1} = (s_{k,l}) + \hat M^{k+1}$. 
Thus we can write $I_k = (s_{k,l})_{l=1}^N$ as a two-sided ideal in $\hat A/\hat M^{k+1}$.

Here we use a following version of Nakayama's lemma.
Let $(A,M) \in (\text{Art}_r)$ and $I$ a two-sided ideal.
Assume that there are elements $h_i \in I$ such that $I = MI + (h_i)$.
Then $I = (h_i)$.
Indeed let $\bar I = I/(h_i) \subset \bar A = A/(h_i)$.
Then $\bar I = M\bar I$.
Since $M$ is nilpotent, $\bar I = M\bar I = \dots = M^m \bar I = 0$ for some $m$.

Now we have a commutative diagram 
\[
\begin{CD}
\hat A/(\hat M\hat I+\hat M^{k+2}) @>>> \hat A/(\hat I+\hat M^{k+2}) \\
@VVV @VVV \\
\hat A/(\hat M\hat I+\hat M^{k+1}) @>>> \hat A/(\hat I+\hat M^{k+1})
\end{CD}
\]
Then the obstruction for the extension on the first line $o_{\xi_{k+1}} = \sum z_l \otimes s_{k+1,l}$ for 
$s_{k+1,l} \in (\hat I+\hat M^{k+2})/(\hat M \hat I+\hat M^{k+2})$
is mapped to $o_{\xi_k} = \sum z_l \otimes s_{k,l}$.
Hence we have $s_{k+1,l} + \hat M^{k+1} = s_{k,l} + \hat M^{k+1}$.
Thus we can define $s_l \in \hat I$ such that $s_l + \hat M^{k+1} = s_{k,l} + \hat M^{k+1}$ for all $k$. 
Then the $s_l$ generate $\hat I$. 
\end{proof}

%%%%%%%%%%%%%%%%%%%%%%%%%%%%%%%%%%%%%%%%%%
%%%%%%%%%%%%%%%%%%%%%%%%%%%%%%%%%%%%%%%%%%
%%%%%%%%%%%%%%%%%%%%%%%%%%%%%%%%%%%%%%%%%%

Graduate School of Mathematical Sciences, University of Tokyo,
Komaba, Meguro, Tokyo, 153-8914, Japan. 

kawamata@ms.u-tokyo.ac.jp

\end{document}